\documentclass{amsart}
\usepackage{amsmath, amssymb,amsthm,latexsym}
\usepackage{enumerate}

\newtheorem{theorem}{Theorem}
\newtheorem{lemma}[theorem]{Lemma}

\newtheorem*{main theorem}{Main Result}

\theoremstyle{definition}

\title[Bisimple $\mathcal{H}$-trivial f.p. congruence-free monoids]{A countable series of bisimple $\mathcal{H}$-trivial finitely presented congruence-free monoids}

\author{Victor~Maltcev}

\address{Department of Mathematics and Statistics, Sultan Qaboos University, Al-Khodh 123, Muscat, 
Sultanate of OMAN}

\email{\texttt{victor.maltcev@gmail.com}}

\keywords{Finitely presented, congruence-free, bisimple, $\mathcal{H}$-trivial}

\begin{document}

\maketitle

It is well-known a countable series of infinite finitely presented simple groups by Graham Higman~\cite{Higman}.
In view of that Boone-Higman Conjecture (see~\cite{LS} to learn about the results on finitely generated simple groups) is still an open problem, every example of finitely generated (and especially finitely presented) group and semigroup are very precious. Surprisingly, the semigroup analogue of simple groups -- congruence-free semigroups -- were studied not that extensively as that was for simple groups. The only results in this area are due to Jean-Camille Birget, we refer the reader to~\cite{Birget1}--\cite{Birget5} to learn about semigroup analogues of Higman's groups.

The author of the current note asked himself -- would it be possible to find the `really' semigroup counterpart of Higman's series, i.e. so that each of the semigroups is finitely presented, congruence-free, bisimple, but yet $\mathcal{H}$-trivial. Indeed such a series exists and the main goal of the note is to prove

\begin{main theorem}\label{th:main}
For any $n\geq 1$ the monoid $M_n$ presented by the (finite)
confluent noetherian system
\begin{eqnarray*}
a^{n+1}ba^nb &\to& 1\\
a^{n+1}ba^{n-1}b &\to& 1\\
&\cdots&\\
a^{n+1}bab &\to& 1\\
a^{n+1}b^2 &\to& b
\end{eqnarray*}
is bisimple, $\mathcal{H}$-trivial and congruence-free.
\end{main theorem}

\begin{proof}
Since $M_n$ is given by a finite complete system, it is convenient
to work with elements of $M_n$ in their normal forms. So, a typical
element of $M_n$ looks like
\begin{equation}\label{eq:form}
b^{\varepsilon}(a^{d_1}b^{k_1})\cdots(a^{d_s}b^{k_s})(a^{n_1}b)\cdots(a^{n_p}b)a^l,
\end{equation}
where $\varepsilon\geq 0$, $s\geq 0$, $p\geq 0$, $l\geq 0$; $1\leq
d_1,\ldots,d_s\leq n$; $k_1,\ldots,k_s\geq 1$; $n_1,\ldots,n_p\geq
n+1$.

Later we will need the following norm for the elements of $M_n$: if
$w\in M_n$ admits the normal form~\eqref{eq:form}, then
\begin{equation*}
\|w\|=\varepsilon+k_1+\cdots+k_s+n_1\cdots+n_p+l.
\end{equation*}

\vspace{\baselineskip}

\noindent\textbf{Bisimplicity and $\mathcal{H}$-triviality}

\vspace{\baselineskip}

It is straightforward from the presentation that $a^l\mathcal{R}1$
for all $l\geq 0$, and that $a^mb\mathcal{R}1$ for all $m\geq n+1$.
Furthermore, $b^{\varepsilon}\mathcal{L}1$ for all $\varepsilon\geq
0$, and $a^db^k\mathcal{L}1$ for all $1\leq d\leq n$ and $k\geq 1$.
Using these, by~\eqref{eq:form} we immediately see that every
element of $M_n$ is $\mathcal{D}$-equivalent to $1$ and so $M_n$ is
bisimple. One can quite easily assure oneself that $M_n$ is, in
addition, $\mathcal{H}$-trivial.

\vspace{\baselineskip}

\noindent\textbf{Congruence-freeness}

\vspace{\baselineskip}

That $M_n$ is congruence-free is equivalent to the following: for
every congruence $\rho$ on $M_n$, and $u\neq v$ in $M_n$ such that
$u\rho v$, it follows that $\rho=M_n\times M_n$.

We induct on $\|u\|+\|v\|$ to show that if $u\neq v$ in $M_n$ and
$\rho$ is a congruence on $M_n$ with $u\rho v$, then $\rho=M_n\times
M_n$. Prior to starting our induction we prove two facts we will use
frequently within the induction arguments:
\begin{lemma}\label{lm:first-lemma}
Any group homomorphic image of $M_n$ is trivial.
\end{lemma}

\begin{proof}
Let $G$ be the group presented by the initial presentation for
$M_n$. Then $a^{n+1}b=1$ and so $a^nb=1$ in $G$. This yields $a=1$
and $b=1$ in $G$.
\end{proof}

\begin{lemma}\label{lm:second-lemma}
Let $\rho$ be a congruence on $M_n$ such that $a^db\mathcal{R}1$ in
$M_n/\rho$ for some $1\leq d\leq n$. Then $\rho=M_n\times M_n$.
\end{lemma}

\begin{proof}
Since $a^{n+1}ba^db=1$, we have that $a^db\in H_1$ in $M_n/\rho$.
Again by $a^{n+1}ba^db=1$, then $a^{n+1}b\in H_1$ and so
$a^{n+1-d}\in H_1$ (in $M_n/\rho$). This yields $a\in H_1$ and thus
$b^2\in H_1$ (in $M_n/\rho$). Therefore $a,b\in H_1$ (in $M_n/\rho$)
and thus $M_n/\rho$ is a group. By Lemma~\ref{lm:first-lemma},
$\rho=M_n\times M_n$.
\end{proof}

\vspace{\baselineskip}

\emph{Base of induction: $\|u\|+\|v\|=1$}

\vspace{\baselineskip}

Without loss we may assume that $\|u\|=1$ and $v=1$. The
possibilities for $u$ are:
\begin{itemize}
\item
$u=b$. Then $b\rho 1$ and so $a^{n+1}\rho 1$ and $a^{n+2}\rho 1$.
Thus $a\rho 1$ and so $\rho=M_n\times M_n$.
\item
$u=a$. Then $a\rho 1$ and so $b^2\rho 1$ and $b^2\rho b$. This
implies $b\rho 1$ and thus $\rho=M_n\times M_n$.
\item
$u=a^db$ with $1\leq d\leq n$. Then by Lemma~\ref{lm:second-lemma},
$\rho=M_n\times M_n$.
\end{itemize}

\vspace{\baselineskip}

\emph{Induction step:
$\bigl(<\|u\|+\|v\|\bigr)\longmapsto\bigl(\|u\|+\|v\|\bigr)$}

\vspace{\baselineskip}

We may assume that $u$ and $v$ are in their normal forms. First let
us sort out the case when one of $u$ and $v$ is the identity. Let,
say, $v=1$. If $u$ starts with $b$, then $b\in H_1$ in $M_n/\rho$.
Then by $a^{n+1}b^2=b$, $a\in H_1$ in $M_n/\rho$. By
Lemma~\ref{lm:first-lemma}, then we have $\rho=M_n\times M_n$. If
$u$ starts with $a$, then three cases can happen:
\begin{itemize}
\item
$u\equiv a^dbU$ for $1\leq d\leq n$. Then $a^db\mathcal{R}1$ in
$M_n/\rho$ and so by Lemma~\ref{lm:second-lemma} $\rho=M_n\times
M_n$.
\item
$u\equiv Ua^mb$ for $m\geq n+1$. Then $ab\rho Ua^mbab=Ua^{m-n-1}$
and so $aba^{n+1}b\rho 1$. Then $ab\mathcal{R}1$ in $M_n/\rho$ and
we are done.
\item
$u\equiv Ua^mba^l$ for $m\geq n+1$ and $l\geq 1$. If $l\geq n$, then
$$a^{n+1}bUa^mba^{l+1}b\rho a^{n+1}bab=1$$ and running the arguments from
the previous case, we derive that $\rho=M_n\times M_n$. If $l\leq
n-1$, then $Ua^mb\rho a^{n+1-l}bab$. This yields
$a^{n+1-l}baba^l\rho Ua^mba^l\rho 1$, and so
$a^{n+1-l}b\mathcal{R}1$ in $M_n/\rho$, thus $\rho=M_n\times M_n$.
\end{itemize}

From now on, in the remainder of the induction step we may assume
that $u\neq 1$ and $v\neq 1$. Now we sort out the case when at least
one of $u$ and $v$ starts with $b$. Essentially there are only three
following cases:
\begin{itemize}
\item
$u\equiv bU$ and $v\equiv bV$. Then $U\not\equiv V$ and so $U\neq V$
(since $U$ and $V$ in their normal forms). But $b\mathcal{L}1$ and
so $U\rho V$. By induction, $\rho=M_n\times M_n$.
\item
$u\equiv b^{\varepsilon}U$ and $v=a^db^kV$ for $1\leq d\leq n$,
$\varepsilon\geq 1$, $k\geq 1$. Then $b^{\varepsilon}U=a^{n+1}b\cdot
b^{\varepsilon}U\rho b^{k-1}V$. Note that $b^{\varepsilon}U$ and
$b^{k-1}V$ are in their normal forms. Thus if
$b^{\varepsilon}U\not\equiv b^{k-1}V$, by induction it follows that
$\rho=M_n\times M_n$. So, assume that $b^{\varepsilon}U\equiv
b^{k-1}V$. Then $k\geq 2$ and so $a^{n+1-d}b^{k-1}V\rho
a^{n+1}b^2\cdot b^{k-2}V=b^{k-1}V$. Now by induction,
$\rho=M_n\times M_n$.
\item
$u\equiv b^{\varepsilon}U$ and $v\equiv a^mbV$ for $\varepsilon\geq
1$ and $m\geq n+1$. Then $b\mathcal{R}1$ in $M_n/\rho$, and so $b\in
H_1$ in $M_n/\rho$. By $a^{n+1}b^2=b$, this gives $a\in H_1$ in
$M_n/\rho$ and so by Lemma~\ref{lm:first-lemma}, $\rho=M_n\times
M_n$.
\end{itemize}

We are left with the situation when both $u$ and $v$ start with $a$.
First we deal with the case when one of $u$ and $v$ starts with a
prefix of the form $a^db$ for some $1\leq d\leq n$. Up to changing
the roles of $u$ and $v$, this splits into the following two cases:
\begin{itemize}
\item
$u\equiv a^{d+r}b^kU$ and $v\equiv a^db^tV$ for $1\leq d,d+r\leq n$,
and $d,r,k,t\geq 1$, and such that $U$ is either empty or starts
with $a$. Premultiplying $u\rho v$ with $a^{n+1}ba^{n-d-r}$, we
obtain $b^{k-1}U\rho b^{t-1}V$. Both $b^{k-1}U$ and $b^{t-1}V$ are
in their normal forms. If $b^{k-1}U\not\equiv b^{t-1}V$, then by
induction we are done. So, assume that $b^{k-1}U\equiv b^{t-1}V$.
Then $a^{d+r}b^kU\rho a^db^kU$. If $k>1$, then by premultiplying the
last $\rho$-equivalence with $a^{n+1-d-r}$, we obtain $b^{k-1}U\rho
a^{n+1-r}b^kU$ and can apply induction. So assume that $k=1$. Then
$a^{n+1}bU\rho a^{n+1-r}bU$ and so in general $a^{n+1+pr}bU\rho
a^{n+1-r}bU$ for all $p\geq 0$. For a sufficiently large $p$,
$a^{n+1+pr}bU\mathcal{R}1$ and so $a^{n+1-r}b\mathcal{R}1$ in
$M_n/\rho$. By Lemma~\ref{lm:second-lemma}, then $\rho=M_n\times
M_n$.
\item
$u\equiv a^db^kU$ and $v\equiv a^mbV$ for $1\leq d\leq n$, $k\geq
1$, $m\geq n+1$; and such that $U$ is either empty or starts with
$a$. If $k\geq 2$, then $b^{k-1}U\rho a^{n+1-d+m}bV$ and so
$b\mathcal{R}1$ in $M_n/\rho$. Then, as above, $a\rho b\rho 1$ and
so $\rho=M_n\times M_n$. If $k=1$, then $a^dbU\rho a^mbV$ and so
$U\rho a^{n+1}ba^mbV$. Hence $a^dba^{n+1}ba^mbV\rho a^mbV$ and so
$a^dba^{n+1}b\rho 1$. Therefore $a^db\mathcal{R}1$ in $M_n/\rho$ and
by Lemma~\ref{lm:second-lemma} we are done.
\end{itemize}

In the remainder of the induction step, both $u$ and $v$ will start
with prefixes of the form $a^mb$ for $m\geq n+1$. If both $u$ and
$v$ end with $a$, we can cancel this distinguished $a$ (since
$a\mathcal{R}1$) and then apply induction. So without loss we will
assume that $u\equiv Ua^mb$ (with $m\geq n+1$). We have two cases to
consider depending on whether $v$ ends with $a$ or not:
\begin{itemize}
\item
$v\equiv Va^pb$ with $p\geq n+1$. Then postmultiplying $Ua^mb\rho
Va^pb$ with $ab$, we obtain $Ua^{m-n-1}\rho Va^{p-n-1}$ and may
apply induction.
\item
$v\equiv Va^l$ with $l\geq 1$. Then $Ua^mb\rho Va^l$ and so, by
postmultiplying with $a^nb$, $Ua^{m-n-1}\rho Va^{l+n}b$. If
$Ua^{m-n-1}\not\equiv Va^{l+n}b$, then we may apply induction. So
assume that $Ua^{m-n-1}\equiv Va^{l+n}b$. Then $m=n+1$ and $u\rho v$
reads as
\begin{equation}\label{eq:something}
Va^{l+n}ba^{n+1}b\rho Va^l.
\end{equation}
If $l\geq n+1$, then $Va^lb\rho Va^{l+n}b^2=Va^{l-1}b$, and since
$Va^lb$ is in its normal form (and regardless whether $Va^{l-1}b$
must be reduced to get its normal form), we may apply induction. So
we will assume that $l\leq n$. If $2\leq l\leq n$, then
postmultiply~\eqref{eq:something} with $a^{n+1-l}$ to obtain
$Va^{n+1}b\rho Va^{l+n}b$ and then use induction. Hence let $l=1$,
then posmultiply~\eqref{eq:something} with $b$: $Vab\rho
Va^{n+1}b^2=Vb$. If $V\equiv 1$, then $ab\rho b$ and so $b^2\rho 1$
and $a\rho ab^2\rho b^2\rho 1$. This yields that $M_n/\rho$ is a
group and so we are done. Thus we may assume that $V$ is non-empty:
$V\equiv V'a^tb$ (with $t\geq n+1$). Then
\begin{equation*}
V'a^{t-n-1}=V'a^tbab\rho V'a^tb^2=V'a^{t-n-1}b.
\end{equation*}
If $t-n-1\geq n+1$, then we may apply induction. So let $t-n-1\leq
n$. If $V'\equiv 1$, then $a^n\rho a^nb$ and so $1\rho a^{n+1}ba^n$.
This implies $a\in H_1$ in $M_n/\rho$, which in turn gives $b\in
H_1$ in $M_n/\rho$. These imply $\rho=M_n\times M_n$. Thus we may
assume that $V'$ is non-empty: $V'=V''a^qb$ (with $q\geq n+1$). Then
$V''a^{q-n-1}\rho V''a^qba^{t-n-1}$. If $t-n-1=0$ we may apply
induction. So assume that $1\leq t-n-1\leq n$ and then by
postmultiplying with $b$ the last $\rho$-equivalence, we obtain
$V''a^{q-n-1}b\rho V''a^{q-n-1}$. Proceeding further in this manner
we eventually will arrive at two distinct $\rho$-equivalent words in
their normal forms and then apply induction.
\end{itemize}
The proof is complete.
\end{proof}


\begin{thebibliography}{3}

\bibitem{Birget1}
J.-C.~Birget, Monoid generalizations of the Richard Thompson groups, \emph{J. Pure Appl. Algebra} {\bf 213} (2009) 264--278.

\bibitem{Birget2}
J.-C.~Birget, The $\mathcal{R}$- and $\mathcal{L}$-orders of the Thompson-Higman monoid $M_{k,1}$ and their complexity, \emph{Internat. J. Algebra Comput.} {\bf 20} (2010) 489--524.

\bibitem{Birget3}
J.-C.~Birget, The Thompson-Higman monoids $M_{k,i}$: the $\mathcal{J}$-order, the $\mathcal{D}$-relation, and their complexity, \emph{Internat. J. Algebra Comput.} {\bf 21} (2011) 1--34.

\bibitem{Birget4}
J.-C.~Birget, Bernoulli measure on strings, and Thompson-Higman monoids, \emph{Semigroup Forum} {\bf 83} (2011) 1--32.

\bibitem{Birget5}
J.-C.~Birget, Monoids that map onto the Thompson-Higman groups, \emph{Semigroup Forum} {\bf 83} (2011) 33--51.

\bibitem{Higman}
G.~Higman, Finitely presented infinite simple groups, \emph{Notes on Pure Mathematics, Department of Pure Mathematics, Department of Mathematics, I.A.S. Australian National University, Canberra} (1974).

\bibitem{LS}
R.~C.~Lyndon, P.~E.~Schupp, Combinatorial group theory, \emph{Classics in Mathematics. Springer-Verlag, Berlin} (2001).

\end{thebibliography}
\end{document}